\documentclass[12pt,reqno]{amsart}
\usepackage[usenames]{color}
\usepackage[latin1]{inputenc} 

\usepackage[english]{babel} 

\usepackage{enumerate}
\usepackage{amssymb, amsmath}
\usepackage{geometry,graphicx}
\usepackage{hyperref}
\usepackage{tikz-cd}

\newtheorem*{theorem*}{Theorem}
\newtheorem*{remark*}{Remark}
\theoremstyle{plain}
\newtheorem{theorem}{Theorem}

\newtheorem{lemma}[theorem]{Lemma} 
\newtheorem{corollary}[theorem]{Corollary}
\newtheorem{proposition}[theorem]{Proposition}

\theoremstyle{definition}
\newtheorem{remark}[theorem]{Remark}
\newtheorem{remarks}[theorem]{Remarks}
\newtheorem{definition}[theorem]{Definition}

\newcommand{\hooklongrightarrow}{\lhook\joinrel\longrightarrow}  
\DeclareMathOperator{\Gal}{Gal}
\DeclareMathOperator{\GL}{GL}
\DeclareMathOperator{\SL}{SL}

\begin{document}
\title{Raising the level at your favorite prime}

\author{Luis Dieulefait}
\thanks{The first author is partially supported by MICINN grant MTM2015-66716-P}

\author{Eduardo Soto}
\thanks{The second author is partially supported by MICINN grant MTM2016-78623-P}

\address{{\sc L. Dieulefait}: Departament de Matem\`atiques i Inform\`atica\\ Universitat de Barcelona, Gran Via de les Corts Catalanes, 585, 08007, Barcelona, Spain}
\email{ldieulefait@ub.edu}
\address{{\sc E. Soto}: Departament de Matem\`atiques i Inform\`atica\\ Universitat de Barcelona, Gran Via de les Corts Catalanes, 585, 08007, Barcelona, Spain}
\email{eduard.soto@ub.edu}

\maketitle

\begin{abstract}
	In this paper we prove a level raising theorem for some weight $2$ trivial
	 character newforms at almost \emph{every} prime $p$. 
	This is done by ignoring the residue characteristic at which the level raising appears. 
\end{abstract}


\section*{Introduction}
	
	For a newform $h$ and a prime $\mathfrak l$ in $\bar{\mathbb Z}$ consider 
	the semisimple $2$ dimensional continuous Galois representation $\bar\rho_{h,\mathfrak l}$
 	with coefficients in 
	$\mathbb F_{\mathfrak l}=\bar{\mathbb Z}/\mathfrak l$
	attached to $h$ and let $\{a_p(h)\}_p\subset \bar{\mathbb Z}$ be the sequence of 
	prime index Fourier coefficients of $h$. 
	Let $f$ and $g$ be newforms of weight $2$ and trivial character. We say that $f$ and 
	$g$ are \emph{Galois-congruent} if there is some
	prime $\mathfrak l$ in $
	\bar{\mathbb Z}$ such that $\bar{\rho}_{f,\mathfrak l}$, $\bar{\rho}_{g,\mathfrak l}$ are 
	isomorphic.	
%
	This is equivalent to 
$$
	a_p(f) \equiv a_p(g) \pmod{\mathfrak l}
$$
	for all but finitely many $p$. 
	In 1990 Ribet proved the following
\begin{theorem*}[K. Ribet]
	Let $f$ be a 
	newform in $S_2(\Gamma_0(N))$ such that the mod $\mathfrak l$ Galois representation
$$
	\bar\rho_{f,\mathfrak l}: \Gal(\overline{\mathbb Q}\mid \mathbb Q)\longrightarrow 
	\GL_2(\mathbb F_\mathfrak l)
$$
	is absolutely irreducible. Let $p\nmid N$ be a prime satisfying
$$
	a_p(f) \equiv \varepsilon(p+1) \pmod {\mathfrak l}
$$
	for some $\varepsilon\in\{\pm1\}$.
	Then there exists a newform $g$ in $S_2(\Gamma_0(pM))$, for some divisor $M$ of $N$ such that 
	$\bar\rho_{f,\mathfrak l}$ is isomorphic to $\bar\rho_{g,\mathfrak l}$. 
	If $2\notin \mathfrak l$ then $g$ can be chosen with $a_p(g)=\varepsilon$.
\end{theorem*}

	Hence under some conditions one can \emph{raise the level} of $f$ at $p$. 
	That is, there is a newform $g$ Galois-congruent to $f$ with level divisible by $p$ once.
	When considering level-raisings of $f$ at $p$ we will tacitly assume that $p$ is 
	\emph{not in the level} of $f$.
	In this paper we do level raising at \emph{every} $p > 2$ by admitting congruences at
	any  prime $\mathfrak l$. 
	More precisely, we prove the following.

\begin{theorem}\label{thm1}
	Let $f$ be a newform in $S_2(N)$ and let $p$ be a prime not dividing $N$.
	Assume that 
	
	$\boldsymbol{(AbsIrr)}$ $\bar\rho_{f,\mathfrak l}$ is absolutely irreducible for every $\mathfrak l$.
	
	$\boldsymbol{(a_2)}$ If $p=2$ assume that $a_2(f)^2 \neq 8$.\\
	Then there exists some $M$ dividing $N$ and some newform $g$ in $S_2(Mp)$ such 
	that $f$ and $g$ are Galois-congruent.
\end{theorem}


	We proof Theorem \ref{thm1} and a variant of it in section \ref{proofs}. 
	A theorem of B. Mazur implies that Theorem \ref{thm1} applies to \emph{most} of the
	rational elliptic curves. We exhibit an infinite family of
	modular forms satisfying $\boldsymbol{(AbsIrr)}$ with coefficient fields constant equal to 
	$\mathbb Q$, see section \ref{section curves}
\begin{remark*}
	It is worth remarking the existence of  infinite families with coefficient fields of unbounded degree satisfying all of the condition $\boldsymbol{(AbsIrr)}$, see \cite{DieulefaitWiese}.
\end{remark*}	
	Lemma \ref{sign} together with Ribet's theorem imply that we can 
	\emph{choose the sign} of $a_p(g)$ when the congruence of $f$ and $g$ is in odd 
	characteristic.  An obstruction appears
	in characteristic $2$ since Ribet's methods identify $1, -1$ mod $2$. Le Hung and Li 
	\cite{LeHungLi} have recently provided a solution to this problem for some $f$ arising
	from elliptic curves using $2$-adic modularity theorems of \cite{Allen} for the ordinary 
	case and quaternion algebras for the supersingular case. In this paper we treat the
	ordinary case.
	
\begin{theorem}\label{thm2}
	Let $f$ be a newform in $S_2(N)$, let $p$ be a prime not dividing $6N$ and choose a sing
	$\varepsilon  \in \{\pm 1\}$. 
	Assume that $f$ satisfies 	$\boldsymbol{(AbsIrr)}$ and for every $\mathfrak l\ni 2$ assume that 
	
	$\boldsymbol{(DiehReal)}$ $\bar\rho_{f,\mathfrak l}$ has dihedral image induced from a 
	real quadratic extension,
	
	$\boldsymbol{(2Ord)}$ $\bar\rho_{f,\mathfrak l}\vert_{G_{2}}\simeq
	\left(
	\begin{array}{cc}
	1	&*\\
		&1	
	\end{array}
	\right)$.\\
	Then there exists some $M$ dividing $N$ and some newform $g$ in $S_2(Mp)$ such 
	that $f$ and $g$ are Galois-congruent and $a_p(g) = \varepsilon$.
\end{theorem}
	
	We deal with this obstruction in section \ref{mod2} following techniques in \cite{LeHungLi}.

%
%

%
%
%
	
	Let $E/\mathbb Q$ be an elliptic curve. Modularity theorems attach to $E$ a newform 
	$f(E)$ such that $\bar{\rho}_{f,\mathfrak l} \simeq E[\ell] \otimes_{\mathbb{F}_\ell}
	\mathbb F_{\mathfrak l}$ modulo semisimplification, for every $\mathfrak l$. We obtain
	an application to elliptic curves.

\begin{theorem}\label{thm3}
	Let $E/\mathbb Q$ be an elliptic curve such that
	\begin{itemize}
		\item$E$ has no rational $q$-isogeny for every $q$ prime,
		\item $\mathbb Q(E[2])$ has degree $6$ over $\mathbb Q$. 
	\end{itemize}
	Let $p$ be a prime of good reduction.
	Then there exists some divisor $M$ of the conductor of $E$ and some newform $g$ in 
	$S_2(Mp)$  such that $f(E)$ and $g$ are Galois-congruent. Let $\varepsilon\in \{\pm1\}$.
	Assume further that 
	\begin{itemize}
	\item $p\geq 5$,
	\item $E$ has good or multiplicative reduction at $2$ and
	\item $E$ has positive discriminant
	\end{itemize} then g can be chosen with 
	$a_p(g) = \varepsilon$.
\end{theorem}


\subsection*{Notation} 
	Let $\bar{\mathbb Q}$ denote the algebraic closure of $\mathbb Q$ in $\mathbb C$. 
	Let $\bar{\mathbb Z}$ be the ring of algebraic integers contained in $\bar{\mathbb Q}$. 
	We use $p, q,\ell$ to denote rational primes and $\mathfrak l$, $\mathfrak l'$ to denote 
	primes of $\bar{\mathbb Z}$. We use \emph{prime} of $\bar{\mathbb Z}$ to refer 
	to maximal ideals of $\bar{\mathbb Z}$, i.e. non-zero prime ideals.
	We denote by $\mathbb F_{\mathfrak l}$ the residue field of 
	$\mathfrak l$ and $\ell$ its characteristic. We consider modular forms as power series 
	with complex coefficients and for a newform $f$ we define by $\mathbb Q_f$ its field of
	coefficients, that is the number field $\mathbb Q_f=\mathbb Q(\{a_p\}_p)$. 
	We denote by $G_{\mathbb Q}$ the absolute Galois group $\Gal(\bar{\mathbb Q}\mid
	\mathbb Q)$ and by $G_p$ a decomposition group of $p$ contained in $G_{\mathbb Q}$.

\numberwithin{theorem}{section}


\section{Newforms}
	Let $\Gamma_0(N)$ be the subgroup of $\SL_2(\mathbb Z)$ corresponding to upper 
	triangular matrices mod $N$. The space $S_2(N):=S_2(\Gamma_0(N))$ of weight $2$
	 level $N$ trivial Nebentypus cusp forms is a finite dimensional vector space over 
	 $\mathbb C$. 
	 For every $M$ dividing $N$,  $S_2(M)$ contributes in $S_2(N)$ under the so-called 
	 degeneracy maps $S_2(M)\hooklongrightarrow S_2(N)$.
	Let $S_2(N)^{old}$ be the subspace of $S_2(N)$ spanned by the images of the degeneracy 
	maps for every $M\mid N$. Let $S_2(N)^{new}$ be the orthogonal space of $S_2(N)^{old}$ 
	with respect to the Peterson inner product. 
	A theorem of Atkin-Lehner says that $S_2(N)^{new}$ admits a basis of Hecke eigenforms 
	called newforms, this basis is unique. 

\subsection{Galois representation} 
	Let $f$ be a newform of level $N$ and let $\mathfrak l$ be a prime of residue characteristic
	 $\ell$. 
	A construction of Shimura (see \cite{DDT} section 1.7) attaches to $f$ an abelian variety 
	$A_f$ over $\mathbb Q$ of dimension $n=[\mathbb Q_f:\mathbb Q]$. 
	$A_f$ has good reduction at primes not dividing $N$. Let $\mathbb Q_{f,\mathfrak l}$ 
	denote the completion of $\mathbb Q_f$ with respect to $\mathfrak l$. 
	Working with the Tate module $\mathcal V_\ell(A_f) = \varprojlim_n A_f[\ell^n]
	\otimes _{\mathbb Z_\ell} \mathbb Q_\ell$ one can attach to $A_f$ a continuous Galois 
	representation 
$$
	\rho_{f,\mathfrak l}:G_\mathbb Q\longrightarrow \GL_2(\mathbb Q_{f,\mathfrak l})
$$
	such that $det \rho_{f,\mathfrak l}$ is the $\ell$-adic cyclotomic character and
	 $tr \rho_{f,\mathfrak l} (Frob_p)=a_p(f)$ for every $p\nmid N\ell$.  
	Indeed,   $\mathcal V_\ell(A_f)$ and $\rho_{f,\mathfrak l}$ are unramified at
	 $p\nmid N\ell$ by Neron-Ogg-Shafarevich criterion.
	 Then $\bar\rho_{f,\mathfrak l}$ is obtained as the semisimple reduction of 
	 $\rho_{f,\mathfrak l}$ mod $\mathfrak l$ tensor $\mathbb F_{\mathfrak l}$. 
	
	Next definition is central in this paper.
\begin{definition}
	Let $f$, $g$ be newforms of weight two, level $N$, $N'$ respectively and trivial character. 
	\begin{itemize}
	\item We say that $f$ and $g$ are 
		\emph{Galois-congruent} if there is some prime $\mathfrak l$ in $\bar{\mathbb Z}$ 
		such that
	$$
		\bar\rho_{f,\mathfrak l}\simeq 	\bar\rho_{g,\mathfrak l}.
	$$
		\item We say that \emph{$g$ is a level-raising of $f$ at $p$} if
		$f$ and $g$ are Galois-congruent and $p\parallel N'$ but
		$p\nmid N$.
		\item We say that \emph{$g$ is a strong level-raising of $f$ at $p$ over $\mathfrak l$}
		if $g$ is a level raising of $f$ at $p$ with $N'= Np$.
	\end{itemize}
\end{definition}

 \begin{remark}
	From Brauer-Nesbitt theorem (\cite{CR} theorem 30.16) we have that a
	semisimple Galois representation $\bar\rho:\Gal(\bar{\mathbb Q}\mid \mathbb Q)
	\rightarrow GL_2(\bar{\mathbb F}_\ell)$ is uniquely determined by the characteristic 
	polynomial function. 
	Hence $\bar{\rho}$ is determined by $tr$ and $det$.
	Since all Galois representations we consider have cyclotomic determinant 
	Galois-congruence is equivalent to congruence on traces of unramified Frobenius 
	(cf. Chebotarev density theorem, \cite{Serre98} Corollary 2). 
\end{remark}
\begin{remark}
	Let $R$ be a ring at which $2$ is invertible and $A$ a free $R$-module of rank $2$. 
	For an endomorphism $f$ of $A$ we have that $tr(f^2) = tr(f)^2 - 2 det(A)$ and hence 
	$det$ is determined by $tr$. This is \cite{DDT} Proposition 2.6 b) for $d=2$
	and gives necessary conditions on existence of Galois-Congruency for general
	weights and levels (cf. \cite{Serre87}).

\end{remark}

\begin{remark}
	With our definition Le Hung and Li \cite{LeHungLi} do strong level raising at a set of primes
	with some extra requirements always in characteristic $2$.
\end{remark}

	Let $\omega_\ell$ denote the mod $\ell$ cyclotomic character and for $\alpha\in
	\bar{\mathbb Z}$ let $\lambda_\alpha$ be the unique unramified character 
	$\lambda_\alpha:G_p \rightarrow \mathbb F_{\mathfrak l}^\times$ sending the arithmetic 
	Frobenius $Frob_p$ to $\alpha$ mod $\mathfrak l$. 
	We collect in the following theorem work of Deligne, Serre, Fontaine, Edixhoven, Carayol
	and Langlands. 
	It gives some necessary conditions for level-raising existence. 

\begin{theorem}\label{local}
	Let $g$ be a newform in $S_2(Mp)$ with $p\nmid 2M$ and fix a prime $\mathfrak l$.
	Then
$$
	\bar\rho_{g,\mathfrak l}\vert_{G_p}\simeq
	\left(
		\begin{array}{cc}
			\omega_\ell	&*\\
						&1
		\end{array}
	\right)\otimes \lambda_{a_p(g)}.
$$
	Let $f$ be a newform $S_2(N)$, $p\nmid 2N$ and fix a prime $\mathfrak l$ containing 
	$p$. Then either $\bar\rho_{f,\mathfrak l}\vert_{G_p}$ is irreducible or 
$$
	\bar\rho_{f,\mathfrak l}\vert_{G_p}\simeq
	\left(
		\begin{array}{cc}
			\omega_p\lambda_{a_p(f)^{-1}}		&*\\
										&\lambda_{a_p(f)}
		\end{array}
	\right)
$$
with $*$ `peu ramifié'.
\end{theorem}

\begin{proof}
	Case $p\in \mathfrak l$ is Theorem 6.7 in \cite{Breuil} for $k=2$. Because $a_p(g)
	\in \{\pm1\}$,
	we have that $g$ is ordinary at $p$.
	Case $p \notin \mathfrak l$ follows from Carayol's theorem in \cite{Carayol}.
	
	The statement for $f$ is Corollary 4.3.2.1 in \cite{BreuilMezard}.
\end{proof}

Following corollaries were inspired by Proposition 6 in \cite{BCDF}.

\begin{corollary}\label{reciprocal}
	Let $f\in S_2(N)$, $g\in S_2(Mp)$ be Galois-congruent newforms over $\mathfrak l$,
	$p\nmid 2NM$. Then
$$
	a_p(f) \equiv a_p(g) (p+1) \pmod {\mathfrak l}.
$$	
\end{corollary}

\begin{proof}
	\emph{Case $p\notin \mathfrak l$}. We have that $tr\bar\rho_{g,\mathfrak l}(Frob_p)
	=a_p(g) ( p+1)$.
	On the other hand $\bar\rho_{f,\mathfrak l}$ is unramified at $p$ and has trace
	$a_p(f)$ at $Frob_p$.\\
	\emph{Case $p\in \mathfrak l$}. The isomorphism 
	$\bar{\rho}_{f,\mathfrak l}\vert_{G_p}\simeq \bar{\rho}_{g,\mathfrak l}\vert_{G_p}$
	implies that $\bar{\rho}_{f,\mathfrak l}\vert_{G_p}$ reduces and we have equality 
	of characters
$$
	\{\omega_p \lambda_{a_p(g)}, \lambda_{a_p(g)}\} =
	 	\{\omega_p \lambda_{a_p(f)^{-1}}, \lambda_{a_p(f)}\} 
$$
	The mod $p$ cyclotomic character is ramified since $p\neq 2$. In particular 
$
	a_p(g) \equiv a_p(f)\pmod{ \mathfrak l}.
$
\end{proof}

	As a consequence we obtain a result on congruent modular forms with level-raising.
\begin{corollary}\label{congGcong}
	Let $f\in S_2(N)$, $g\in S_2(Mp)$ be newforms, $p\nmid 2NM$. If $f$ and $g$ are 
	\emph{congruent}, that is 
$$
	a_n(f)\equiv a_n(g)\pmod{ \mathfrak l}\qquad \text{ for every $n$,}
$$ 
	then 
$$
	\qquad \ell=p \qquad \text{ and } \qquad a_p(f)\equiv a_p(g) = \pm 1 \pmod{\mathfrak l}. 
$$
\end{corollary}

\begin{proof}
	Congruency implis Galois-congruency. We have that 
	$$
	a_p(g)\equiv a_p(f) \equiv a_p(g) (p+1)\pmod {\mathfrak l}.
	$$ 
	The corollary follows since $a_p(g)\in \{\pm 1\}$.
\end{proof}
	\subsection{Ribet's level raising}\label{sectionRibet}
	Ribet's theorem says that the necessary condition of Corollary \ref{reciprocal} for 
	level-raising turns out to be enough modulo some irreducibility condition.

\begin{theorem}[Ribet's level raising theorem]\label{Ribet}
	Let $f$ be a newform in $S_2(N)$ such that the mod $\mathfrak l$ Galois representation
$$
	\bar\rho_{f,\mathfrak l}: \Gal(\overline{\mathbb Q}\mid \mathbb Q)\longrightarrow 
	\GL_2(\mathbb F_\mathfrak l)
$$
	is absolutely irreducible. Let $p\nmid N$ be a prime satisfying
$$
	a_p(f) \equiv \varepsilon(p+1) \pmod {\mathfrak l}
$$
	for some $\varepsilon\in\{\pm1\}$.
	Then there exists a newform $g$ in $S_2(pM)$, for some divisor $M$ of $N$ such that 
	$\bar\rho_{f,\mathfrak l}$ is isomorphic to $\bar\rho_{g,\mathfrak l}$. 
	If $2\notin \mathfrak l$  $g$ can be chosen with $a_p(g)=\varepsilon$.
\end{theorem}

\begin{remarks}
	\begin{itemize}
		\item Ribet's original approach deals with modular Galois representations $\bar{\rho}$
		 so that in particular there is some newform $f$ such that $\bar{\rho}\simeq
		 \bar{\rho}_{f,\mathfrak l}$. 
		His approach deals with traces of Frobenii, this forces him to deal with unramified 
		primes only, hence the hypothesis $p\neq \ell$. 
		As he explains later the theorem can be stated in terms of Hecke operators and 
		hence in terms of Fourier coefficients even if $p=\ell$, when $f$ is $p$-new.
	
		\item Every normalized Hecke eigenform $f'$ has attached a unique 
		newform $f$
		so that the $\mathfrak l$-adic Galois representations attached to $f'$ are the ones 
		attached to $f$.
		Furthermore, the level of $f$ divides the level of $f'$. 
		Hence $p$-new in Ribet's article means new of level $pM$ for some $M\mid N$.
	\end{itemize}
\end{remarks}

	We introduce a definition in order to deal with the irreducibility condition.
\begin{definition}
	Let $f\in S_2(N)$ be a newform, let $\mathfrak l$ be a prime of $\overline{ \mathbb Z}$ and
	let $\bar\rho_{f,\mathfrak l}:\Gal(\overline{\mathbb Q}\mid\mathbb Q)\rightarrow
	\GL_2(\overline{\mathbb F}_{\mathfrak l})$ denote the semisimple mod $\mathfrak l$ Galois
	representation attached to $f$ by Shimura. 
	We say that $f$ satisfies condition $\boldsymbol{(AbsIrr)}$ if $\bar\rho_{f,
	\mathfrak l}$ is absolutely irreducible for every prime $\mathfrak l$. 
\end{definition}

	In section \ref{examples} we provide explicit examples with $\mathbb Q_f =  \mathbb Q$.
	See \cite{DieulefaitWiese} Theorem 6.2 for a construction of families $\{f_n\}$ for which
	the set of degrees $\{\dim_\mathbb Q \mathbb Q_{f_n}\}_n$ is unbounded.


\section{Bounds and arithmetics of Fourier coefficients}

%
	
	In light of Ribet's theorem and the following well-known properties of Fourier coefficients 
	we study some  arithmetic properties of $p+1\pm a_p$. 
	
\begin{theorem}\label{Hasse}
	Let $f$ be a newform in $S_2(N)$ and let $a_p$ be the $p$-th Fourier coefficient of $f$, 
	$p$ prime. Then
\begin{enumerate}[(i)]
		\item $a_p\in \overline{\mathbb Z}$,
		\item $a_p$ is totally real. That is, its minimal polynomial splits over $\mathbb R$,	
		\item(Hasse-Weil Theorem) $|\sigma(a_p)|\leq 2\sqrt p$ for every embedding 
		$\sigma:\mathbb Q(a_p)\hooklongrightarrow \mathbb R$.
	\end{enumerate}
\end{theorem}

	We say that $a_p\in \overline{\mathbb Q}$ is a \emph{$p$-th 
	Fourier coefficient} if $a_p$ satisfies conditions $(i)-(iii)$. 
		
\subsection{Arithmetic lemmas}
	Let $K$ be a number field with ring of integers $\mathcal O$. 
	Let $S$ be the set of complex embeddings $\sigma:K\rightarrow \mathbb C$ of $K$.
	For every $\alpha \in \mathcal O$ we consider the norm 
$$
	N_K(\alpha) = \prod_{\sigma\in S} \sigma(\alpha)
$$	
	and the characteristic polynomial
$$
	P_{\alpha}(X) = \prod_{\sigma\in S} X-\sigma(\alpha).
$$
	One has that $P_\alpha (0) = (-1) ^{|S|}N_K(\alpha)$. Following well-known lemma
	is basic algebraic number theory.
		
\begin{lemma}
	Let $K$ be a number field and $\alpha$ an algebraic integer of $K$.
	Then $P_\alpha(X)\in \mathbb Z[X]$ and $N_K(\alpha)\in
	\mathbb Z$. A rational prime $\ell$ 
	divides $N_K(\alpha)$ if and only if there is some prime $\mathfrak l$ in
	$\bar{\mathbb Z}$ of residue characteristic $\ell$ such that $\alpha\equiv 0
	\pmod{\mathfrak l}$.
	In particular $\alpha$ is a unit of $\mathcal O$ if and only if $N_K(\alpha)=\pm1$.
\end{lemma}

\begin{proof}
	Let $n= |S| = \dim_{\mathbb Q} K$. Consider the embedding $\iota: K
	\hooklongrightarrow End_{\mathbb Q}(K)$, where $\iota(\alpha)$ is the 
	multiplication-by-$\alpha$ morphism. 
	A choice of integral basis of $K$
	induces an embedding $\iota: K\hooklongrightarrow M_{n\times n}(\mathbb Q)$ 
	with $\iota(\mathcal O) \subset M_{n\times n}(\mathbb Z)$. Then $P_\alpha(X)$
	is the characteristic polynomial of $\iota(\alpha)$ by Cayley-Hamilton theorem. 
	For a nonzero $\alpha\in \mathcal O$ one has that $|N_K(\alpha)| = \left| \mathcal O
	/\alpha \mathcal O\right| =\prod_i |\mathcal O/\mathfrak p_i^{e_i}|$ where 
	$\alpha\mathcal O = \prod_i \mathfrak p_i ^{e_i}$ is the factorization in prime ideals 
	of the ideal generated by $\alpha$. Thus $\ell$ divides $N_K(\alpha)$ if and only if
	some prime $\mathfrak p$ of $\mathcal O$ containing $\ell$ divides
	$\alpha\mathcal O$. Moreover, the map $MaxSpec(\bar{\mathbb Z})\rightarrow 
	MaxSpec(\mathcal O)$ given by $\mathfrak l\mapsto \mathfrak l \cap \mathcal O$ 
	is surjective and the lemma follows.
\end{proof}

\begin{lemma}\label{sign}
	Let $a_p$ be a $p$-th Fourier coefficient. 
	\begin{enumerate}[(a)]
		\item $(p+1+a_p)(p+1-a_p)$ is unit in $\bar{\mathbb Z}$ if and only if $p=2$ and 
		$a_2^2 = 8$.
		\item If $p>3$ then $p+1\pm a_p$, is not unit in $\bar{\mathbb Z}$.
	\end{enumerate}	
\end{lemma}
\begin{proof}
	Let $K = \mathbb Q(a_p)$ and $S=\{\sigma\in :\mathbb Q(a_p)\hooklongrightarrow 
	\mathbb R\}$ be the set of embeddings. 
	Its cardinality equals $n$ the degree of $\mathbb Q(a_p)\mid \mathbb Q$.
	\begin{enumerate}[(a)]
		\item
		We have
		\begin{align*}
			\sigma((p+1+a_p)(p+1-a_p)) &= (p+1)^2 - \sigma(a_p)^2\\
					&\geq (p+1)^2 - 4 p\\
					&=(p-1)^2\\
					&\geq 1.
		\end{align*}
		Hence $N_K((p+1+a_p)(p+1-a_p))\geq (p-1)^{2n}\geq 1$. Equalities hold if 
		and only if $p=2$ and $9-a_2^2 = 1$.
		\item Since $\sigma(p+1 \pm a_p) \geq p+1- 2\sqrt p = (\sqrt p -1) ^2$ then
		$$
		N_K(p+1\pm a_p)\geq (\sqrt p -1) ^{2n}>1
		$$
		provided that $p>3$.
	\end{enumerate}
\end{proof}


\begin{lemma}[Avoiding $p$]\label{potential}
	Fix a positive odd integer $n$. There exists an integer $C_n$ such that if $a_p$
	 is a $p-th$ coefficient of degree $n$  with $p>C_n$ then there is a
	prime $\mathfrak l$ not over $p$ such that
$$
	\mathfrak l \mid (a_p + p+1)(a_p-p-1).
$$
\end{lemma}

\begin{proof}
	Let $K=\mathbb Q(a_p)$ and assume that $(p+1+a_p)(p+1-a_p)$ factors as product of
	primes over $p$ in the ring of integers of $K$. 
	Then $N_K(p+1 -a_p)$, $N_K(p+1+a_p)$ are powers of $p$ in
	the  closed interval  
	$I=[(\sqrt p -1)^{2n}, (\sqrt p + 1)^{2n}]$. We can take $p$ great enough so that $p^n$ 
	is the unique power of $p$ in $I$. Thus
	\begin{align*}
		N_K(p+1-a_p) 	&=N_K(p+1+a_p)= p^n\\
		N_K(-p-1-a_p) 	&=(-1)^n N_K(p+1+a_p) = -p^n
	\end{align*}
	In particular $0\equiv P_{a_p}(p+1)-P_{a_p}(-p-1) = 2p^n\pmod{p+1}$.
\end{proof}

	We can describe the bound $C_n$: conditions $p^{n +1}, p^{n-1}\notin ~I$ are equivalent to

	\begin{align*}
		p	&>\left(\dfrac{p}{p - 2\sqrt p +1}\right)^n=:\alpha(p,n),\\
		p	&> \left(\dfrac{p+2\sqrt p+1}{p}\right)^n=:\beta(p,n)
	\end{align*}
	Notice that $\beta<\alpha$ and that $p$ satisfies $p>\alpha(p,n)$ if and only if 
	$x^n >x^{n-1} +1$ where $x^{2n}= p$. Since $n$ is odd we can take $\theta$ the 
	greatest real root of $X^n- X^{n-1} -1$ and $C_n:=\theta^{2n}$. Notice that $C_n/n^2$ 
	has finite limit.

\begin{lemma}\label{rationalavoid}
	The best bound for $n=1$ is $C_1=2$.
\end{lemma}

\begin{proof}
	Notice that $(p+1)^2 - a^2_p=1$ if $p=2$, $a_p= \pm 1$.
	Following the notation  above we have that $\theta=2$ for $n=1$ and $C_1=4$ works. 
	Thus it is enough to check that $(4 -a_3)(4 +a_3)$ is not $\pm$ a power of $3$. Both
	factors are positive by Hasse's bound. Thus $4+a_3=3^a$, $4-a_3=3^b$ and
	$3^a + 3 ^b = 8$.
\end{proof}

\section{Proofs}\label{proofs}
\subsection{Proof of main result and variant}
\begin{proof}[Proof of Theorem \ref{thm1}]
	Let $f\in S_2(N)$ new and $p\nmid N$. We need to check that Ribet's theorem applies for
	some $\mathfrak l$. By Lemmas \ref{Hasse} and \ref{sign}, $(p+1+a_p)(p+1-a_p)$ is not 
	invertible in $\bar{\mathbb Z}$. 
	Hence it is contained in a maximal ideal $\mathfrak l$. 
	That is, either $a_p\equiv p+1\pmod{\mathfrak l}$ or $a_p\equiv -p-1\pmod{\mathfrak l}$.
\end{proof}
	Following variant allows us to do level-raising at $p$ over characteristic $\ell\neq p$.
	This together with Corollary \ref{congGcong} ensures that the predicted Galois-congruency
	is not a cogruence of \emph{all} Fourier coefficients, at least when the level-raising is at
	$p\neq 2$.
	
\begin{theorem}\label{avoid}
	Let $f$ be a newform in $S_2(N)$ such that $n:=\dim_{\,\mathbb Q} K_f$ is odd.
	Assume that 
	
	$\boldsymbol{(AbsIrr)}$ $\bar\rho_{f,\mathfrak l}$ is absolutely irreducible for every $\mathfrak l$.
	
	There exists a constat $C>0$ such that for every prime $p>C$ $f$ has a level-raising $g$ at $p$
	over a prime $\mathfrak l$ of residue characteristic different from $p$. $C$ depens only on
	$n$.
\end{theorem}

\begin{proof}
	Let $f\in S_2(N)$ new. Due to $\boldsymbol{(AbsIrr)}$ it is enough to find a maximal ideal
	$\mathfrak l$ not over $p$. This is done in Lemma \ref{potential}.
\end{proof}

\subsection{Choice of sign mod $2$}\label{mod2}

	In this section we adapt some ideas of \cite{LeHungLi} to our case.
	The strategy is to solve a finitely ramified deformation problem. This kind of deformation
	problem consists on specifying the ramification behavior at all but one chosen prime $q$.  
	If such a deformation problem has solution and some modularity theorem applies
	this provides newforms with specified weight, character and prime-to-$q$ part level.
	If one chooses an auxiliary prime $q$, a twist argument kills the ramification at $q$
	so that one recovers a newform with the specified weight, character and level.
	
	Fix a prime ideal $\mathfrak l\ni 2$ of $\bar{\mathbb Z}$.  Let $\rho_2$ be
	a Galois representation $G_{\mathbb Q}\rightarrow \SL_2(\mathbb F_{\mathfrak l})$
	with \emph{dihedral image} $D$ and let $E=\bar{\mathbb Q}^D$ be the number field 
	fixed by $\ker \rho_2$. The order of an element in $\SL_2(\mathbb F_{\mathfrak l})$
	is either $2$ or odd. This forces $D$ to have order $2r$, $2\nmid r$. In particular 
	$E\mid\mathbb Q$ has a unique quadratic subextension $K\mid \mathbb Q$ and $\rho_2$ is
	induced from a character 
	$
	\chi: \Gal(\bar{\mathbb Q}\mid K)\rightarrow \mathbb F_{\mathfrak l}^\times
	$ of order $r$.
	
	We say that $q$ is an \emph{auxiliary prime} for $\rho_2$ if
\begin{itemize}
	\item $q\equiv 3 \pmod 4$ and
	\item $\rho_2$ is unramified at $q$ and $\rho_2(Frob_{q})$ is non-trivial of odd order.
\end{itemize}

\begin{proposition}\label{twist}
	Let $g$ be a newform in $S_2(Mq^\alpha)$, $q\nmid M$, such that $\bar{\rho}_{g,\mathfrak l}$
	 is unramified at an auxiliary prime $q$. Then either $g$ or $g\otimes \genfrac(){}{}{\cdot}{q}$
	 has level $M$.
\end{proposition}

\begin{proof}
	$\bar{\rho}_{g,\mathfrak l}(Frob_{q})$ has different eigenvalues by the order condition, 
	thus $\rho_{g,\mathfrak l}\vert_{I_q}$ factors through a quadratic character $\eta$ of 
	$I_q$ (Lemma 3.4 in \cite{LeHungLi}). By the structure of tame inertia at $q\neq 2$ there 
	is a unique open subgroup in $I_q$ of index $2$ and $\eta:I_q\twoheadrightarrow 
	\Gal(\mathbb Q_q^{ur}(\sqrt q)\mid \mathbb Q_q^{ur})\simeq \{\pm 1\}$. 
	If $\eta$ is trivial then $\alpha=0$ and we are done. 
	Otherwise, $\eta$ extends locally to $G_q \twoheadrightarrow \Gal(\mathbb Q_q(\sqrt{-q})
	\mid	\mathbb Q_q)$ and globally to the Legendre symbol
$$
	\genfrac (){}{0}{\cdot}{q}: G_\mathbb Q\twoheadrightarrow 
	\Gal(\mathbb Q(\sqrt{-q})\mid \mathbb Q)\simeq \{\pm 1\}.
$$
	Legendre symbol over $q$ is only ramified at $q$ and the proposition follows.
\end{proof}

	 Auxiliary primes are inert at $\mathbb Q(i)$ and split at $K$ by a parity argument.
	 In particular, $\rho_2$ has auxiliary primes only if $K\neq \mathbb Q(i)$.

\begin{lemma}
	Let $\rho_2:G_{\mathbb Q}\rightarrow \SL_2(\mathbb F_{\mathfrak l})$ be a Galois 
	representation as above. Assume that $\rho_2$ is not ramified at $p$ and that
	$K\neq \mathbb Q(i)$. Then the set of auxiliary primes for $\rho_2$ splitting at 
	$\mathbb Q(\sqrt p)$ has positive density in the set of all primes. 	
\end{lemma}

\begin{proof}
	As in Lemma 3.2 of \cite{LeHungLi} $E$ and $\mathbb Q(i, \sqrt p)$ are linearly disjoint
	since $E$ is unramified at $p$ and $K\neq \mathbb Q(i)$. Chebotarev density theorem 
	implies the lemma.
\end{proof}

\begin{theorem}
	Let $f$ be a newform in $S_2(N)$, $p$ be a prime not dividing $6N$ and $\varepsilon
	\in \{\pm 1\}$ a sign.
	Assume that $a_p\equiv 1+p$ mod $\mathfrak l$ for some prime $\mathfrak l$
	containing $2$.
	Assume that
\begin{enumerate}
	\item
	$\bar\rho_{f,\mathfrak l}$ has dihedral image induced from a 
	real quadratic extension, and 
	\item
	$\boldsymbol{(2Ord)}$ $\bar\rho_{f,\mathfrak l}\vert_{G_{2}}\simeq
	\left(
	\begin{array}{cc}
	1	&*\\
		&1	
	\end{array}
	\right)$.
\end{enumerate}
	then there exists some $M$ dividing $N$ and some newform $g$ in $S_2(Mp)$ such that
	$f$ and $g$ are Galois-congruent and $a_p(g) = \varepsilon$.
\end{theorem}

\begin{proof}
	Let $q$ be an auxiliary prime for $\bar\rho_{f,\mathfrak l}\vert_{G_2}$ splitting 
	at $\mathbb Q(\sqrt p)$. 
	By Theorem 4.2 of 
	\cite{LeHungLi} there is some newform $g$ in $S_2(Npq^\alpha)$ with
	$a_p(g) = \varepsilon$.  Let $g'$ be the newform in $S_2(Np)$ obtained by Proposition 
	\ref{twist}. Then $a_p(g') = a_p(g)$ since $\genfrac (){}{}{p}{q}=1$.
\end{proof}

\begin{proof}[Proof ot Theorem \ref{thm2}]
	By Lemma \ref{sign} there are some maximal ideals $\mathfrak l^+$, $\mathfrak l^-$
	such that 
	$a_p(f) \equiv p+1 \pmod{\mathfrak l^+}$
	 and 
	 $a_p(f) \equiv -p-1 \pmod{\mathfrak l^-}$.
	 If $2\notin \mathfrak l^+, \mathfrak l^-$ then Ribet's theorem applies and we are done.
	 Otherwise apply previous theorem.
\end{proof}


\section{Case $n=1$. Elliptic curves and $\mathbb Q$-isogenies}\label{section curves}
	Let $E/\mathbb Q$ be an elliptic curve and let $E_p/\mathbb F_p$ be the mod $p$ reduction of 
	(the Néron model of) $E$ for a prime $p$. 
	Consider the integer $c_p= p+1-\#E_p$. 
	Then there is a unique newform $f$ of weight $2$ such that $a_p(f) = c_p$ for every prime $p$. 
	This is a consequence of modularity of elliptic curves over $\mathbb Q$. 
	In particular,  $\bar{\rho}_{f,\mathfrak l}$ and $E[\ell]\otimes \mathbb F_{\mathfrak l}$ are
	isomorphic up to semisimplification for every prime $\mathfrak l$. 
	In this section we characterize elliptic curves whose corresponding newform $f$ satisfies 
	$\boldsymbol{(AbsIrr)}$. 

	Let $E/\mathbb Q$ be an elliptic curve, $\ell$ an odd prime and $c\in\Gal(\mathbb C\mid 
	\mathbb R)\subset \Gal(\bar{\mathbb Q}\mid\mathbb Q)$ be the complex conjugation. 
	Then $c$ acts on $E[\ell]$ with characteristic polynomial $X^2 -1$. 
	This follows from the existence of Weil pairing. In particular $E[\ell]$ is 
	irreducible if and only
	if $E[\ell]\otimes \mathbb F_{\mathfrak l}$ is irreducible. 
	We say that $E$ satisfies $\boldsymbol{(Irr)}$ if $E[\ell]$ is irreducible for every $\ell$.  
	From a particular study of the case $\ell=2$ one obtains the 
\begin{lemma}
	Let $E/\mathbb Q$ be an elliptic curve. Then $E$ satisfies $\boldsymbol{(AbsIrr)}$ if and only 
	$E$ satisfies $\boldsymbol{(Irr)}$ and $\mathbb Q(E[2])$ has degree 6 over $\mathbb Q$.
\end{lemma}

\subsection{Isogenies}
	In practice one can deal with $\boldsymbol{(Irr)}$ by studying the graph of isogeny
	classes. LMFDB project has computed in \cite{LMFDB}  a huge amount of elliptic curves and 
	isogenies.  We recall some well known results on this topic.

	Let $E, E'$ be elliptic curves defined over $\mathbb Q$. 
	An \emph{isogeny} is a nonconstant morphism $E\longrightarrow E'$ of abelian varieties over 
	$\mathbb Q$. The map
$$
	\begin{array}{ccc}
		\{E\rightarrow E'\text{ isogeny}\}/{\cong}	&\longrightarrow 	&\{\text{finite $\mathbb Z[\Gal(\bar{\mathbb Q}\mid \mathbb Q)]$-submodules of $E$} \}\\
		\varphi							&\longmapsto 		&Ker\,\varphi
	\end{array}
$$
	defines a bijection. 
	Hence, the torsion group $E[n]$ corresponds to the multiplication-by-n map 
	$E\overset{[n]}{\rightarrow} E$ under the bijection. 
\begin{lemma}

	Let $E/\mathbb Q$ be an elliptic curve. The following are equivalent
	\begin{enumerate}
		\item $E$ satisfies $\boldsymbol{(Irr)}$.
		\item the graph of isogeny classes of $E$ is trivial.
		\item every finite $\mathbb Z[\Gal(\bar{\mathbb Q}\mid\mathbb Q)]$-submodule of $E$ is
		 of the form $E[n]$ for some $n$.
	\end{enumerate}
\end{lemma}

\begin{proof}
	Let $G_{\mathbb Q}$ denote the absolute Galois group $\Gal(\bar{\mathbb Q}\mid \mathbb Q)$. 
	We will prove that $(1)\Rightarrow (2) \Rightarrow (3) \Rightarrow (1)$. 
	For the first implication let $E\overset{\varphi}{\rightarrow} E'$ be an isogeny. 
	Then there exists a maximal $n$ such that $\varphi$ factors as 
$$
	E\overset{[n]}{\longrightarrow} E \overset{\psi}{\longrightarrow} E'
$$
	for some isogeny $\psi:E\rightarrow E'$. It can be checked that $n$ is the biggest integer 
	satisfying $E[n] \subset Ker\, \varphi$.
	Let $r=\# Ker\, \psi$. If $p\mid n$ then $E[p]\cap Ker\, \psi$ is a nontrivial subrepresentation of 
	$E[p]$. 
	If $E$ satisfies $\boldsymbol{(Irr)}$ then $r=1$ and $\psi$ is an isomorphism.\\
	For the second let $H$ be a finite $\mathbb Z[G_{\mathbb Q}]$-submodule of $E$. 
	It corresponds to some isogeny $E\overset{\varphi}{\rightarrow} E'$ with kernel equal to $H$. 
	By hypothesis $E$ and $E'$ are isomorphic say $E'\overset{h}{\rightarrow} E$. 
	Thus $h\circ \varphi$ is an endomorphism of $E$ defined over $\mathbb Q$ and hence 
	$h\circ \varphi = [n]$ for some $n$. The last implication is trivial.
\end{proof}

	If the isograph is unkown one can still do something. In 1978 Barry Mazur proved 
	(see \cite{Mazur}) the 

\begin{theorem}[B. Mazur]
	Let $E/\mathbb Q$ be an elliptic curve and let $\ell$ be a prime such that $E[\ell]$ is reducible. 
	Then 
$$
	\ell\in \mathcal T:=\{2, 3, 5, 7, 11, 13, 17, 19, 37, 43, 67, 163\}.
$$
\end{theorem}

	Hence, there is a complete list of possible irreducible submodules of $E[\ell]$.
	We will use Mazur's theorem later  in order to exhibit a family of elliptic curves
	satisfying $\boldsymbol{(Irr)}$.

\subsection{Twists}

	Condition $\boldsymbol{(AbsIrr)}$ is invariant under 
	$\bar{\mathbb Q}$-isomorphism. This follows from the fact that Galois
	representations attached to $\bar{\mathbb Q}$-isomorphic rational
	elliptic curves differ by finite character. 
	The useful invariant in this context is the $j$-invariant. More precisely, the map
$$
	\begin{array}{cccc}
		j: &Ell:=\{E/\mathbb Q \text{ elliptic curve }\}/\cong_{\bar{\mathbb Q}}&\longrightarrow &\mathbb Q\\
		&E								&\longmapsto 	&j(E)
	\end{array}
$$
	is a bijection, hence $\boldsymbol{(AbsIrr)}$ is codified in the $j$-invariant.
	
\begin{definition}
	Let $a/b\in \mathbb Q$, with $a,b$ coprime integers. The \emph{Weil height} of $a/b$ is
	$$
	h(a/b) := \max \{ |a|, |b|\}.
	$$
	Let $S$ be a subset of $Ell$. We say that $S$ has \emph{Weil density} $d$
	if
	$$
		\lim_{n\rightarrow \infty}\dfrac{\#\{E\in S: h(j(E))\leq n\}}{\#\{x\in \mathbb Q: h(x)
		\leq n\} }=d.
	$$
	
\begin{proposition}\label{density}
	Let $S$ be the set elliptic curves satisfying $\boldsymbol{(AbsIrr)}$ modulo isomorphism.
	Then $S$ has Weil density $1$.
\end{proposition}

\begin{proof}
	$j$-invariant morphism extends to an isomorphism $X(1)_{\mathbb Q} \rightarrow
	\mathbb P^1_{\mathbb Q}$ of rational algebraic curves. Here $X(1)_{\mathbb Q}$ 
	denotes a rational model of the trivial-level modular curve. Hence $Ell$ is the set 
	$Y(1)(\mathbb Q)\subseteq X(1)(\mathbb Q)$ of rational non-cuspidal points of $X(1)$.
	Let $p\geq 2$ be prime and let $X_0(p)_{\mathbb Q}$ a model over $\mathbb Q$ of
	the modular curve of level $\Gamma_0(p)$. We have the forgetful map
	$X_0(p)\rightarrow X(1)$ which is a morphism of algebraic curves of degree 
	$p+1$. Elliptic curves not satisfying 
	$\boldsymbol{(Irr)}$ correspond to non-cuspidal points in the image of 
	$f_p:X_0(p)(\mathbb Q) \rightarrow X(1)(\mathbb Q)$, for $p\leq 163$ by Mazur.
	Either $X_0(p)$ has genus $0$, in which case $p\in \{2,3,5,7,13\}$, or $X_0(N)$ has
	positive genus, in which case $p\in\{11, 17, 19,	37,43, 67, 163\}$ and 
	$X_0(p)(\mathbb Q)$ is finite. Image of $f_p$ has $0$ Weil density in $X(1)$ for every
	$p\leq 163$, this follows from Theorem $B.2.5$ in \cite{HindrySilverman} for the
	genus $0$ case.
	In particular elliptic curves satisfying $\boldsymbol{(Irr)}$ have density $1$. 
	One can deal similarly with the condition 
	$\dim_{\mathbb Q}\mathbb Q(E[2]) = 6$. 
\end{proof}
\end{definition}
	 
%


\section{Examples}\label{examples}
	\subsection{A family of elliptic curves}
	In this section we give a family of elliptic curves over $\mathbb Q$ satisfying 
	$\boldsymbol{(AbsIrr)}$. 
	First we find a family of elliptic curves with irreducible $2$-torsion as 
	$\bar{\mathbb F}_2[G_\mathbb Q]$-module. 
	This is done by exhibiting a family of rational cubic polynomials with symmetric Galois group. 
	Second we take a subfamily with irreducible 
	$\ell$-torsion as $\mathbb F_\ell[G_\mathbb Q]$-module, for every $\ell\in \mathcal T$. 

\begin{lemma}\label{polirred}
	Let $n\neq \pm1$ be integer such that $3n$ is not square. The polynomial 
	$P_n(X)=X^3 - 3 (n+1)X + 2(n+1)$ 
	has Galois group isomorphic to $S_3$. 
\end{lemma}

\begin{proof}
	Let us see that $P_n$ is irreducible over $\mathbb Q$ when $n\neq 0,\pm1$. 
	Consider a factorization $P_n(X)=(X-a)(X^2 + b X + c)$ over the integers. 
	By equating coefficients we have that
$$
	\begin{cases}
		a				&=b\\
		2a^2 + 3 a c - 2c 	&=0\\
		-ac 				&= 2(n+1)
	\end{cases}
$$
	The conic 
	$0=18 (2 X^2 + 3 X Y-2Y)=(6X + 9Y + 4)(6X - 4)+16$
	 has finitely many integer points, namely 
$$
	(a,b)\in \{(0,0), (-2,1), (1,-2), (2,-2)\}.
$$
	In particular $P_n$ is irreducible if and only if $n \notin\{ -1,0,1\}$. 
	In this case either $P_n$ has Galois group of order three or $P_n$ has Galois group 
	isomorphic to $S_3$, the latter corresponds to the nonsquare discriminant case. 
	The discriminant of $P_n$ is $\Delta_n=3n\cdot36(n+1)^2$ and the lemma follows.
\end{proof}


\begin{lemma}
	Consider the elliptic curve defined over $\mathbb F_{1427}$ given by the equation
$$
	\bar E:Y^2 = X^3 + 3\cdot 11 X - 2\cdot 11.
$$
	Then $\bar{E}[\ell]$ is irreducible over $\mathbb F_\ell$ for every $\ell \in \mathcal T$.
\end{lemma}

\begin{proof}
	It can be checked that $\# \bar E = 1424$. Let $\varphi$ denote the Frobenius over $1427$, 
	then $\varphi$ satisfies
$$
	\varphi^2  - 4 \varphi + 1427=0
$$
	as an endomorphism of $\bar E$. The polynomial  $X^2 - 4X + 1427$ is irreducible over 
	$\mathbb F_\ell$ for every $\ell\in \mathcal T$ and hence $\bar E[\ell]$ is irreducible.
\end{proof}

\begin{theorem}
	Let $n$ be an integer such that
$$
	k\equiv -11 \pmod{1427}.
$$
	Then the elliptic curve given by the equation 
$$
	E_k:Y^2 = X^3 - 3 k X + 2k
$$
	satisfyes $\boldsymbol{(AbsIrr)}$. In particular it is Galois-congruent to infinitely many 
	newforms.
\end{theorem}

\begin{proof}
	Since $-12$ is not a square in $\mathbb F_{1427}$ Lemma \ref{polirred} applies 
	and since $E_k[\ell]$ is unramified over $1427$ for every $\ell \in \mathcal T$
	the theorem follows.
\end{proof}
\begin{remark}
	Notice that Theorem \ref{avoid} together with Lemma \ref{rationalavoid} say that every 
	level-raising of $E_k$ at $p>2$ can be done far from $p$. 
	This together with Corollary \ref{congGcong} implies that
	odd level-raisings of $E_k$ can be chosen not congruent.
\end{remark}
\subsection{Control of $M$}
	Let $f$ be a newform of level $N$ and let $\mathfrak l $ be a prime. If 
	$\bar N =N(\bar\rho_{f,\mathfrak l})$ 
	denotes the  prime-to-$\ell$ conductor of $\bar\rho_{f,\mathfrak l}$ then $\bar N\mid N$. 
	With this in mind we manage in next theorem to take $M=N$. 
	
\begin{theorem}\label{disc}
	Let $E/\mathbb Q$ be an elliptic curve such that
	\begin{enumerate}[(i)]
		\item $E$ has trivial graph of isogeny classes,
		\item $\mathbb Q(E[2])$ has degree $6$ over $\mathbb Q$, 
		\item $E$ is semistable with good reduction at $2$,
		\item $\Delta(E)$ is square-free.
	\end{enumerate}
	Let $N$ denote the conductor of $E$ and let $p\nmid N$ be a prime.
	Then there exists some newform $g\in S_2(Np)$ Galois-congruent to $f(E)$.
\end{theorem}

\begin{proof}
	Let $\mathfrak l$ be a prime and $g$ a newform in $S_2(Mp)$ such that
	$g$ is a level
	raising of $E$ over $\mathfrak l$. Let us prove that $M=N$. 
	Since $\Delta(E)$ is squre-free then $E[\ell]$ is ramified at every prime $p\mid N$,
	$p\neq \ell$, and the prime-to-$\ell$ conductor $N_\ell$ of $E$ is the prime-to-$\ell$
	conductor of $E[\ell]$ (cf. Proposition 2.12 in \cite{DDT}). In particular 
$$
	M\in \{N, N/\ell\}.
$$
	Assume that $M\neq N$, then $N=M\ell$ and $\ell\neq 2$ since $E$ has
	good reduction at $2$. Theorem \ref{local} (or Tate's $p$-adic uniformization) says that $E[\ell]\vert_{G_\ell}$ is
	reducible. In particular 
	$$
	E[\ell]\vert_{G_\ell} \simeq \bar\rho_{f,\mathfrak l}\vert_{G_\ell}\simeq
	\left(
		\begin{array}{cc}
			\omega_\ell\lambda_{a_\ell(f)^{-1}}		&*\\
											&\lambda_{a_\ell(f)}
		\end{array}
	\right)
	$$ 
	with $*$ `peu ramifié'. This together with Proposition 8.2 of \cite{Edixhoven} and
	Proposition 2.12 of \cite{DDT} leads to a contradiction.
\end{proof}
	
\begin{remarks}	
	\begin{itemize}
		\item Condition $(iii)$ is equivalent to $N$ being odd and square-free.
		\item The rational elliptic curve of conductor $43$ satisfies coditions $(i)-(iv)$.
	\end{itemize}
\end{remarks}
	
\section{An application: safe chains}
	When considering safe chains as in \cite{DieulefaitSymmetric} (Steinberg) level-raising
	at an appropriate (small) prime is a useful tool. In particular, this combined with a standard modular
	congruence gives an alternative way of introducing a ``MGD" prime to the level.
	Having a MGD prime in the level is one of the key ingredients in a safe chain.
	Therefore, one could expect to use generalizations of Theorem \ref{thm1} to build safe chains in
	more general settings. In the process of doing so one can rely on  
	tools as in \cite{DieulefaitWiese} to ensure that the condition 
	$\boldsymbol{(AbsIrr)}$ holds when required.

\section*{Acknowledgements}  
	We would like to thank Samuele Anni, Sara Arias-de-Reyna, Roberto Gualdi,
	Xavier Guitart and Xavier Xarles for helpful conversation and comments.

\end{document}